\documentclass[12pt]{article}
\usepackage{amsmath,amssymb,color}
\usepackage{xspace}

\newcommand{\Var}{{\rm{Var}_{\mathbb{C}}}}

\def\1{\underline{1}}

\def\LLL{{\mathbb L}}
\def\Z{{\mathbb Z}}

\newtheorem{theorem}{Theorem}

\newtheorem{lemma}{Lemma}
\newtheorem{proposition}{Proposition}

\newenvironment{definition}
{\smallskip\noindent{\bf Definition\/}:}{\smallskip\par}

\newenvironment{examples}
{\smallskip\noindent{\bf Examples\/}.}{\smallskip\par}

\newenvironment{proof}
{\noindent{\bf Proof\/}.}{{ $\square$}\smallskip\par}
\newenvironment{Proof}
{\noindent{\bf Proof\/}}{{ $\square$}\smallskip\par}

\title{Higher order orbifold Euler characteristics for compact Lie group actions.
\footnote{
group actions, orbifold Euler characteristic, wreath products, generating series.}
}

\author{S.M.~Gusein-Zade \thanks{Partially supported by
the grants
RFBR--13-01-00755, NSh--5138.2014.1.
Address: Moscow State University, Faculty
of Mathematics and Mechanics, GSP-1, Moscow, 119991, Russia. E-mail:
sabir\symbol{'100}mccme.ru} \and I.~Luengo \and
A.~Melle--Hern\'andez \thanks{The authors are partially
supported by the grant MTM2010-21740-C02-01. 
Address: 
ICMAT (CSIC-UAM-UC3M-UCM). Dept.\ of \'Algebra, 
Facultad de Ciencias Matem\'aticas, Universidad Complutense de Madrid, 
28040, Madrid, Spain. \qquad \qquad \qquad
E-mail: iluengo\symbol{'100}mat.ucm.es, amelle\symbol{'100}mat.ucm.es}}
\date{}
\begin{document}
\def\eps{\varepsilon}

\maketitle

\begin{abstract}
We generalize the notions of the orbifold Euler characteristic and of 
the higher order orbifold Euler characteristics to spaces with actions of a compact Lie group.
This is made using  the integration with respect to the Euler characteristic  instead of the summation over finite sets. 
We show that the equation for the generating series of the $k$-th order    orbifold Euler characteristics
of the Cartesian products of the space with the wreath products actions proved by H.~Tamanoi for finite group actions 
and by C.~Farsi and Ch.~Seaton for compact Lie group actions with finite isotropy subgroups holds in this case as well. 
 \end{abstract}

\section{Introduction}\label{sec1}
Let $X$ be a topological space (good enough, say, a quasi-projective variety) 
with an action of a finite group $G$.
For a subgroup $H$ of $G$, let $X^H=\{x\in X: Hx=x\}$ be the fixed point set of $H$. The orbifold Euler characteristic
$\chi^{orb}(X,G)$ of the $G$-space $X$ is defined, e.g., in \cite{AS}, \cite{HH}:
\begin{equation}\label{chi-orb}
 \chi^{orb}(X,G)=
\frac{1}{\vert G\vert}\sum_{{(g_0,g_1)\in G\times G:}\atop{\\g_0g_1=g_1g_0}}\chi(X^{\langle g_0,g_1\rangle})
=\sum_{[g]\in {G_*}} \chi(X^{\langle g\rangle}/C_G(g))\,,
\end{equation}
where $G_*$ is the set of the conjugacy classes of the elements of $G$, $C_G(g)=\{h\in G: h^{-1}gh=g\}$
is the centralizer of $g$, and $\langle g\rangle$ and $\langle g_0,g_1\rangle$ 
are the subgroups generated by the corresponding elements.
Here and bellow we use the additive Euler characteristic, i.e. the one defined through
the cohomologies with compact support.

The higher order Euler characteristics of $(X,G)$ 
(alongside with some other generalizations) were defined
in \cite{AS}, \cite{BF}, \cite{T}.

\begin{definition}
 The {\em orbifold Euler characteristic} $\chi^{(k)}(X,G)$ {\em of order} $k$ of the $G$-space $X$ is
\begin{equation}\label{chi-k-orb}
 \chi^{(k)}(X,G)=
\frac{1}{\vert G\vert}\sum_{{{\bf g}\in G^{k+1}:}\atop{g_ig_j=g_jg_i}}\chi(X^{\langle {\bf g}\rangle})
=\sum_{[g]\in G_*} \chi^{(k-1)}(X^{\langle g\rangle}, C_G(g))\,,
\end{equation}
where ${\bf g}=(g_0,g_1, \ldots, g_k)$, $\langle{\bf g}\rangle$ is the subgroup 
generated by $g_0,g_1, \ldots, g_k$, and
$\chi^{(0)}(X,G)$ is defined as $\chi(X/G)$.
\end{definition}

The usual orbifold Euler characteristic $\chi^{orb}(X,G)$
 is the orbifold Euler characteristic of order $1$, $\chi^{(1)}(X,G)$.

A generalization of this definition  for the orbifold Hodge-Deligne polynomial (for $k=1$) was
 introduced by V.~Batyrev in \cite{baty}.
A ``motivic version'' of it, the higher order generalized Euler characteristics with values 
in the ring $K_0(\Var)[\LLL^{1/m}]$, where  $K_0(\Var)$ is the Grothendieck ring of 
complex quasi-projective varieties, $\LLL$ is the class of the complex affine line,  $m$ 
runs through positive integers, was defined in \cite{GLMHigher}.

Let $G^n=G\times\ldots\times G$ be the Cartesian product of a
group $G$. The symmetric group $S_n$ acts on $G^n$ by permutation
of the factors: $s (g_1, \ldots, g_n) = (g_{s^{-1}(1)} , \ldots,
g_{s^{-1}(n)})$. The {\em wreath product} $G_n = G \wr S_n$ is
the semidirect product of the groups $G^n$ and $S_n$ defined by
the described action. Namely the multiplication in the group $G_n$
is given by the formula $({\bf g}, s)({\bf h}, t) = ({\bf g}\cdot s({\bf h}), st)$, 
where
${\bf g},\, {\bf h} \in G^n$, $s,\, t \in S_n$. The group $G^n$ is a normal
subgroup of the group $G_n$ via the identification of ${\bf g} \in G^n$
with $({\bf g},1) \in G_n$. For a space $X$ with a $G$-action, there
is the corresponding action of the group $G_n$ on the Cartesian
product $X^n$ given by the formula
$$
( (g_1, \ldots, g_n), s)(x_1, \ldots, x_n) = (g_1 x_{s^{-1} (1)},
\ldots, g_n x_{s^{-1} (n)})\,,
$$
where $x_1, \ldots, x_n \in X$, $g_1, \ldots, g_n \in G$, $s\in
S_n$. One can see that (at least for compact $G$) the quotient $X^n/G_n$ is naturally
isomorphic to the symmetric power $S^n (X/G)= (X/G)^n/S_n$ of the quotient $X/G$. 
A formula for the generating series of the $k$-th order orbifold Euler characteristics of the pairs $(X^n, G_n)$
in terms of the $k$-th order  orbifold Euler characteristic of the $G$-space $X$
was given in \cite{T} (see also \cite{BF}): see Theorem 1 bellow.

A generalization of this formula (for $k=1$) for the orbifold Hodge-Deligne polynomial (for finite $G$) was
given in  \cite{WZ}. The corresponding ``motivic'' version can be found in \cite{GLMSteklov}.  
A version of it for the generalized Euler characteristic of order $k$ was formulated in \cite{GLMHigher}.

A ``non-finite version`` of the sum over a finite set is the integral with respect to the Euler 
characteristic: \cite{Viro}, see also \cite{GZ-Survey}. Here we show that this notion permits  to define 
analogous of the orbifold Euler characteristics of order $k$ for a space $X$ with an action of a compact
Lie group $G$ and prove that the equation from \cite{T} for the generating series of the 
  orbifold Euler characteristics of order $k$ of the wreath products holds in this case as well.
The case when all the isotropy subgroups of the $G$-action are finite was studied
 in \cite{FS}.  (There one used another definition not appropiate for actions with non-finite isotropy subgroups.
It was also assumed that the $G$-space was a manifold. However this was connected with the fact that the authors worked  in the framework 
of the orbifold theory.)

It appeared that the first equations in the definitions  (\ref{chi-orb}) and 
(\ref{chi-k-orb}) were less convenient for the proofs of the formulae for the generating series in 
\cite{T} and \cite{WZ} and were not appropiate for the definition of their motivic version
in   \cite{GLMHigher}. In what follows we shall use the second equations as the base for the generalization.

%%%%%%%%%%%%%%%%%%%%%%%%%%%%%%%%%%%%%%%%%%%%%%%%%%%%%%%%%%%%%%%%%%%%%%%%%%%
\section{Orbifold Euler characteristic of order $k$ for actions of compact Lie groups}\label{uno}
%%%%%%%%%%%%%%%%%%%%%%%%%%%%%%%%%%%%%%%%%%%%%%%%%%%%%%

Let $X$ be a topological space (good enough, say, a quasi-projective variety) 
endowed with an action of a compact Lie group $G$. We assume that the action
of $G$ on $X$ has finitely many orbit types and moreover, the space of  
orbits of a fixed orbit type is good enough so that its Euler characteristic makes sense. 
For example  this holds if the action (i.e. the map $G\times X\to X$) is an algebraic one. 
(Each compact Lie group is a real algebraic manifold.)
Let $G_*$ be the space of the conjugacy classes of the elements of $G$. 
The space $G_*$ is a finite CW-complex. 

\begin{definition}
 The {\em orbifold Euler characteristic} of a $G$-space $X$ (i.e. of the pair $(X,G)$) is 
\begin{equation}\label{chi-orb-lie}
 \chi^{orb}(X,G)
:=\int\limits_{{G_*}} \chi(X^{\langle g\rangle}/C_G(g))\, d\chi.
\end{equation}
\end{definition}

\begin{definition}
 The {\em  orbifold Euler characteristic of order $k$} of a $G$-space $X$ (i.e. of the pair $(X,G)$) is 
\begin{equation}\label{chi-k-orb-lie}
 \chi^{(k)}(X,G)
:=\int\limits_{{G_*}} \chi^{(k-1)}(X^{\langle g\rangle},C_G(g))\, d\chi,
\end{equation}
where  $\chi^{(0)}(X,G)$ is $\chi(X/G)$.
\end{definition}

The orbifold Euler characteristic  (\ref{chi-orb-lie}) is the orbifold Euler characteristic of order $1,$
$ \chi^{(1)}(X,G).$ 

For a closed subgroup $H\subset G$, let $X^{(H)}$ be the set of points $x$ in $X$ with the isotropy subgroup 
$G_x=\{g\in G: gx=x \}$ coinciding with $H$, let $X^{([H])}$ be the set of points $x$ with the isotropy subgroup 
conjugate to  $H$. The additivity of the orbifold Euler characteristic of order $k$ with respect to a partitioning of the space 
into $G$-invariant parts implies the following statement.
 \begin{proposition}  One has
\begin{equation*}
  \chi^{(k)}(X,G)= \sum\limits_{[H] \in \mbox{\rm Conjsub\,} G} \chi(X^{([H])}/G) \chi^{(k)}(G/H,G)\,.
\end{equation*}
\end{proposition}

\begin{examples}
 (1)  $\chi^{(k)}(S^1/\Z_m,S^1)=m\chi^{(k-1)}(S^1/\Z_m,S^1)=m^k$;  and, for $k>0$,  $\chi^{(k)}(S^1/S^1,S^1)=0$.

(2) The conjugacy classes of the elements of the group $O(2)$ are $[T_\alpha]$, $0\leq \alpha\leq \pi$, and $[S]$ 
where $T_\alpha\in SO(2)$ is the rotation by the angle $\alpha$, $S\in O(2)\setminus SO(2)$ is the symmetry
with respect to a line.  The centralizer of $T_\alpha$   is $O(2)$ for $\alpha=0,\ \pi$ and $SO(2)$ for $0< \alpha< \pi$.
Therefore one has: for $m$ odd 
\begin{eqnarray*}
& & \chi^{(k)}(O(2)/\Z_m,O(2))=\chi^{(k-1)}(O(2)/\Z_m,O(2))+ \frac{m-1}{2} \chi^{(k-1)}(O(2)/\Z_m,SO(2))\\
&=&
\chi^{(k-1)}(O(2)/\Z_m,O(2))+ (m-1) m^{k-1}=m^k;
\end{eqnarray*}
for $m$ even 
\begin{eqnarray*}
& & \chi^{(k)}(O(2)/\Z_m,O(2))=2\chi^{(k-1)}(O(2)/\Z_m,O(2))+ \frac{m-2}{2} \chi^{(k-1)}(O(2)/\Z_m,SO(2))\\
&=&
2\chi^{(k-1)}(O(2)/\Z_m,O(2))+ (m-2) m^{k-1}=m^k;
\end{eqnarray*}
\begin{eqnarray*}
& & \chi^{(k)}(O(2)/SO(2),O(2))=2\chi^{(k-1)}(O(2)/SO(2),O(2))- \chi^{(k-1)}(O(2)/SO(2),SO(2))\\
&=&
2\chi^{(k-1)}(O(2)/SO(2),O(2))=2^k.
\end{eqnarray*}
\end{examples}

\section{Generating series of the orbifold  Euler characteristics of the wreath products}

To prove the equation (\ref{Principal}) for the generating series of the orbifold Euler characteristics of order $k$
for the Cartesian products of a $G$-space with the wreath products actions, we shall use two technical lemmas (cf. \cite{T},  \cite{GLMHigher}).

\begin{lemma} 
Let $X$ and $X'$ be two spaces with actions of compact Lie groups $G'$ and $G''$ respectively.
Then $X'\times X''$ is a $G'\times G''$-space and one has:
 \begin{equation}
  \chi^{(k)}( X'\times X'', G'\times G'')=\chi^{(k)}( X', G')\cdot \chi^{(k)}( X'', G'')\,.
 \end{equation}
\end{lemma}

The proof is obvious.

\begin{lemma}\label{lemma3} {\rm (cf. \cite[Lemma~4-1]{T})} 
Let $X$ be a $G$-space and let $c$ be an element of the centre of $G$ acting trivially on $X$.  
Let $G\cdot\langle a\rangle $ be the group generated by $G$ and
the additional element $a$ commuting with
all the elements of $G$ and such that $\langle a\rangle \cap G=\langle c\rangle $, $c=a^r$. The space
$X$ can be regarded as a $(G\cdot\langle a\rangle)$-space if one assumes that $a$ acts trivially on $X$.
In the described situation one has
$$
\chi^{(k)}(X, G\cdot\langle a\rangle )=r^k \cdot \chi^{(k)}( X, G)\,.
$$
\end{lemma}

\begin{proof}
We shall use the induction on $k$. For $k=0$ this is obvious (since $\chi^{(0)}(X, G)=\chi(X/G)$).
Each conjugacy class of the elements from $G\cdot\langle a\rangle $ is of the form $[g]a^s$, where $[g]\in G_*$,
$0\le s< r$. The fixed point set of $ga^s$ coincides with $X^{\langle g\rangle}$, i.e. 
$X^{\langle ga^s \rangle}=X^{\langle g \rangle}$ 
(since $a$ acts trivially). The centralizer
$C_{G\cdot\langle a\rangle }(ga^s)$ is $C_G(g)\cdot\langle a\rangle $. Therefore
\begin{eqnarray*}
\chi^{(k)}(X, G\cdot\langle a\rangle )&
=&\int\limits_{{(G\cdot\langle a\rangle)_*}} \chi^{(k-1)}(X^{\langle ga^s \rangle},C_{G\cdot\langle a\rangle }(ga^s))\, d\chi  \\
&=&r \int\limits_{G_*} \chi^{(k-1)}(X^{\langle g \rangle},C_{G}(g)\cdot\langle a\rangle)    \, d\chi  \\ 
&=&r \cdot  r^{k-1}\int\limits_{G_*} \chi^{(k-1)}(X^{\langle g \rangle},C_{G}(g))    \, d\chi =r^k \chi^{(k)}(X, G)\,.
\end{eqnarray*}

\end{proof}

\begin{theorem}\label{main}
One has
\begin{equation}\label{Principal}
\sum_{n\ge 0}\chi^{(k)}(X^n, G_n)\cdot t^n
=\left(\prod\limits_{r_1, \ldots,r_k\geq 1}\left(1- t^{r_1r_2\cdots r_k}\right)^{r_2r_3^2\cdots r_k^{k-1}}\right)
^{-\chi^{(k)}(X, G)}\,.
\end{equation}
\end{theorem}
\begin{Proof}
 The proof will use the induction on the order $k$ similar to the ones in \cite{T}, \cite{GLMHigher}. 

For $k=0$ one has  $\chi^{(0)}( X, G)=\chi(X/G)$, $X^n/G_n \cong S^n(X/G)$ and the equation 
(\ref{Principal}) is a particular case (for $Y=X/G$) of the well-known Macdonald formula  (\cite{mac}):
$$
\sum_{n\geq 0} \chi (S^n Y)\cdot t^n
=(1-t)^{-\chi(Y)}\,.
$$

Suppose that the statement holds for the  orbifold Euler characteristic of order ($k-1$).

%Surprisingly we have not found a proof which uses the integration with respect to the Euler characteristic
%directly, but have used the corresponding integral sums instead of the integrals themselves. 

Let $A_q:=\{[g]\in G_*: \chi^{(k-1)}( X^{\langle g \rangle}, C_G(g))=q\}$. One has $G_*=\bigsqcup_{q} A_q.$
Due to our assumptions only finitely many subspaces $A_q$ are not empty.
According to the definition (\ref{chi-k-orb-lie}),  $\chi^{(k)}( X, G)=\sum_{q} q \chi(A_q)$. 

One has: 
$$
\sum_{n\ge 0}\chi^{(k)}(X^n, G_n)\cdot t^n
=\sum_{n\ge 0} \int\limits_{(G_n)_*}   \chi^{(k-1)}( {(X^n)}^{\langle ({\bf g},s) \rangle}, C_{G_n}(({\bf g},s)))       d\chi  \cdot t^n\,.
$$

A description of the conjugacy classes $[({\bf g},s)]$ in $G_n$ can be found, e.g., in \cite{T}.
The conjugacy class of an element $a=({\bf g},s)\in G_n$ (${\bf g}=(g_1,\ldots, g_n)$, $s\in S_n$)
is completely characterized by its type. Let $z=(i_1,\ldots,
i_r)$ be one of the cycles in the permutation $s$. The {\em
cycle-product} of the element $a$ corresponding to the cycle $z$
is the product $g_{i_r}g_{i_{r-1}}\ldots g_{i_1}\in G$. (The
conjugacy class of the cycle-product is well-defined by the
element ${\bf g}$ and the cycle $z$ of the permutation $s$.) For $[c]\in G_*$
and $r\ge 1$, let $m_r(c)$ be the number of the $r$-cycles in the
permutation $s$ whose cycle-products belong to  $[c]$. 
(There are finitely many pairs $(c,r)$ with $m_r(c)\ne 0$.)
One has
$$
 \sum\limits_{[c] \in G_*, r\geq 1} r m_r (c) = n\,.
$$
The collection  $\{m_r(c)\}_{r,c}$ (or, what is the same, the map $G_*\to \Z_{\geq 0}^\infty, 
[c]\mapsto (m_1(c),m_2(c),\ldots)$) is called the {\em type} of the element $a=({\bf g},s)\in G_n$.
Two elements of the group $G_n$ are conjugate to each other if and only if they are of the same type.

For an element $a=({\bf g},s)\in G_n$, let the number of different $c\in G_*$ which are cycle products  of $a$ be equal to $\ell$.
One has a map from $\bigsqcup_{n\geq 0} (G_n)_*$ to $\bigsqcup_{\ell \geq 0} ((G_*^\ell \setminus \Delta)/S_\ell)$,
 where $\Delta$ is the ``large diagonal" in $G_*^\ell$, i.e. the space of points   $(c_1,\ldots, c_\ell)\in G_*^\ell$
with at least two coinciding components. The restriction of this map to $(G_n)_*$ has finite preimages of points. 
However the preimage of each point is the set $(\Z^\infty\setminus \{{\bf 0}\})^\ell$ consisting in all possible finite sequences 
$(m_1(c_i), m_2(c_i),\ldots),$ for $i=1,\ldots,\ell$, such that for each $i$ the sequence $m_1(c_i), m_2(c_i),\ldots$ is different from 
${\bf 0}=(0,0,\ldots)$. Let us denote the  set of inappropriate sequences $\{m_r(c_i)\}$ by $''\underbar{0}''$.
Thus
$$
\sum_{n\ge 0}\chi^{(k)}(X^n, G_n)\cdot t^n
=\sum_{\ell\geq0} \frac{1}{\ell!}\int\limits_{G_*^\ell\setminus \Delta}  \sum\limits_{\{m_r(c_i)\}\setminus {{}''\underbar{0}{}''}} 
 \chi^{(k-1)}( {(X^n)}^{\langle ({\bf g},s) \rangle}, C_{G_n}(({\bf g},s)))       \cdot t^{\sum_{i,r} r m_r (c_i) }\,d\chi 
$$
where $({\bf g},s)$    is a representative of the conjugacy class of the elements of $G_n$ ($n=\sum_{i,r} r m_r (c_i) $) with the type defined by
$(c_1,\ldots,c_\ell)\in  {G_*}^\ell \setminus \Delta$ and by the sequences $\{m_r(c_i)\}$. In \cite{T} it was shown that the space 
$ (X^n)^{\langle ({\bf g},s) \rangle}$ is canonically isomorphic to the product
\begin{equation}\label{fixed}
 \prod\limits_{i=1}^\ell \prod\limits_{r\geq 1} (X^{\langle c_i \rangle})^{m_r(c_i)} 
\end{equation}
and the centralizer $C_{G_n}(({\bf g},s))$ is isomorphic to the product
\begin{equation*}
\prod\limits_{i=1}^\ell \prod\limits_{r\ge 1}\left\{(C_G(c_i)\cdot\langle a_{r,c_i}\rangle ) \wr S_{m_r(c_i)}\right\}
\end{equation*}
where the factors in the group act on 
 the product (\ref{fixed}) component-wise, 
$C_G(c_i)\cdot\langle a_{r,c_i}\rangle $ is the group generated by
$C_G(c_i)$ and an element $a_{r,c_i}\in G_n $ commuting with all the elements of $C_G(c_i)$ and such that
$a_{r,c_i}^r=c_i$, $\langle a_{r,c_i}\rangle \cap C_G(c_i)=\langle c_i\rangle $,  $a_{r,c_i}$ acts 
on $(X^{\langle c_i \rangle})^{m_r(c_i)}$ trivially.
Therefore 
\begin{eqnarray*}
&& \sum_{n\ge 0}\chi^{(k)}(X^n, G_n)\cdot t^n
\\
&=&\sum\limits_{\ell\geq0} \frac{1}{\ell!}\int\limits_{G_*^\ell\setminus  \Delta}  \sum 
 \chi^{(k-1)} \left(\prod\limits_{i=1}^\ell \prod\limits_{r\geq 1}{(X^{\langle c_i \rangle})}^{m_r(c_i)}, \prod\limits_{i=1}^\ell \prod\limits_{r\geq 1}
\left\{(C_G(c_i)\cdot\langle a_{r,c_i}\rangle ) \wr S_{m_r(c_i)}\right\}
\right)  \cdot t^{\sum_{i,r} r m_r (c_i) } d \chi\, \\
&=&\sum\limits_{\ell\geq0} \frac{1}{\ell!}\int\limits_{G_*^\ell\setminus  \Delta}  \sum
 \prod\limits_{i=1}^\ell \prod\limits_{r\geq 1}
 \chi^{(k-1)} \left({(X^{\langle c_i \rangle})}^{m_r(c_i)}, 
\left\{(C_G(c_i)\cdot\langle a_{r,c_i}\rangle ) \wr S_{m_r(c_i)}\right\}
\right)  \cdot t^{\sum_{i,r} r m_r (c_i) } d \chi\, \\
&=&\sum\limits_{\ell\geq0} \frac{1}{\ell!}\int\limits_{G_*^\ell\setminus  \Delta} 
 \prod\limits_{i=1}^\ell \left[ \prod\limits_{r\geq 1} \left( \sum\limits_{m=0}^\infty  
 \chi^{(k-1)} \left({(X^{\langle c_i \rangle})}^{m}, 
(C_G(c_i)\cdot\langle a_{r,c_i}\rangle ) \wr S_{m}\right) t^{rm}
\right) -1\right] d \chi\,, 
\end{eqnarray*}
where the sums in the second and in the third lines are over 
${\{m_r(c_i)\}\setminus ''\underbar{0}''}$, for each $i=1,\ldots,\ell$ the summand $1$ 
is substracted since not all $m_r(c_i)$ should be equal to zero.  
Using the induction one has
\begin{eqnarray*}
&& \sum_{n\ge 0}\chi^{(k)}(X^n, G_n)\cdot t^n\\
&=&\sum\limits_{\ell\geq0} \frac{1}{\ell!}\int\limits_{G_*^\ell\setminus  \Delta}  \prod\limits_{i=1}^\ell \left[ \prod\limits_{r\geq 1}
\left(\left(\prod\limits_{r_1,\ldots,r_{k-1}\geq 1}
   (1-  t^{rr_1\cdots r_{k-1}})
\right)^{r_2\cdot r_3^2\cdots  r_{k-1}^{k-2}  }\right)^{- \chi^{(k-1)} \left(X^{\langle c_i \rangle}, 
C_G(c_i)\cdot\langle a_{r,c_i}\rangle \right)} -1 \right] d \chi  \\
&=&\sum\limits_{\ell\geq 0} \frac{1}{\ell!}\int\limits_{G_*^\ell\setminus  \Delta}
\prod\limits_{i=1}^\ell \left[ \prod\limits_{r\geq 1}
\left(\left(\prod\limits_{r_1,\ldots,r_{k-1}\geq 1}
   (1-  t^{r r_1\cdots r_{k-1}})
\right)^{r_2\cdot r_3^2\cdots  r_{k-1}^{k-2}  }\right)^{- r^{k-1}\chi^{(k-1)} \left(X^{\langle c_i \rangle}, 
C_G(c_i)  \right)}-1 \right] d \chi \,.
\end{eqnarray*}
One has 
\begin{equation*}
 G_*^\ell\setminus  \Delta= \coprod\limits_{\{ \ell_q \}: \sum  \ell_q =\ell} \frac{\ell !} {\prod \ell_q!} \prod_{q}\left( A_q^{\ell_q}\setminus \Delta \right) \,,
\end{equation*}
where the coefficient   $\frac{\ell !} {\prod \ell_q!}$ is the number of possible decompositions  of $\ell$ elements into groups of sizes $\ell_q$.
Therefore
\begin{eqnarray*}
&& \sum_{n\ge 0}\chi^{(k)}(X^n, G_n)\cdot t^n\\ 
&&\hspace{-1cm}=
\prod\limits_{q}
\sum\limits_{\ell_q\geq 0} \frac{1}{\ell_q!}\int\limits_{A_q^{\ell_q}\setminus  \Delta}
 \prod\limits_{i=1}^{\ell_q} \left[ \prod\limits_{r\geq 1}
 \left(\left(\prod\limits_{r_1,\ldots,r_{k-1}\geq 1}
   (1-  t^{r r_1\cdots r_{k-1}})
\right)^{r_2\cdot r_3^2\cdots  r_{k-1}^{k-2}  }\right)^{- r^{k-1}\chi^{(k-1)} \left(X^{\langle c_i \rangle}, 
C_G(c_i)  \right)}-1 \right] d \chi \\
&&\hspace{-1cm}=
\prod\limits_{q}\left(\sum\limits_{\ell_q=0}^\infty \frac{ \chi(A_q)(\chi(A_q)-1)\cdots(\chi(A_q)-\ell+1)  }{\ell_q !}
\left[ 
\prod\limits_{r_1,\ldots,r_{k}\geq 1}
\left(   (1-  t^{r_1\cdots r_{k}})^{r_2\cdot r_3^2\cdots  r_{k-1}^{k-2}}
\right)^{- r_k^{k-1}q }-1
\right]^{\ell_q}
\right) \\
&=& \prod\limits_{q}\left(
\prod\limits_{r_1,\ldots,r_{k}\geq 1}
   (1-  t^{r_1\cdots r_{k}})^{-r_2\cdot r_3^2\cdots  r_{k}^{k-1}q}
\right)^{\chi(A_q)} 
= \left(
\prod\limits_{r_1,\ldots,r_{k}\geq 1}
   (1-  t^{r_1\cdots r_{k}})^{r_2\cdot r_3^2\cdots  r_{k}^{k-1}}
\right)^{-\sum\limits_{q}q\chi(A_q)} \\
&=& \left(
\prod\limits_{r_1,\ldots,r_{k}\geq 1}
   (1-  t^{r_1\cdots r_{k}})^{r_2\cdot r_3^2\cdots  r_{k}^{k-1}}
\right)^{-\chi^{(k)}(X,G)} \,.
\end{eqnarray*}
In the middle we use the standard formula $\sum\limits_{\ell} \frac{M(M-1)\cdots(M-\ell+1) }{\ell !} T^\ell =(1+T)^M.$
\end{Proof}


\begin{thebibliography}{15}

\bibitem{AS} M.~Atiyah, G.~Segal. 
On equivariant Euler characteristics. 
J. Geom. Phys. 6 (1989), no.4, 671--677. 

\bibitem{baty} V.~Batyrev. Non-Archimedian integrals and stringy
Euler numbers of log-Terminal pairs. J. Eur. Math. Soc. 1 (1999), 5--33.

   


\bibitem{BF} J.~Bryan, J.~Fulman.
Orbifold Euler characteristics and the number of commuting $m$-tuples in the symmetric groups. 
Ann. Comb. 2 (1998), no.1, 1--6. 


\bibitem{FS} C.~Farsi, Ch.~Seaton. 
Generalized orbifold Euler characteristics for general orbifolds and wreath products. 
Algebr. Geom. Topol. 11 (2011), no.1, 523--551.


\bibitem{GZ-Survey} S.M.~Gusein-Zade.
 Integration with respect to the Euler characteristic and its applications.
 Russian Math. Surveys 65 (2010), no.3, 399--432.



\bibitem{GLMSteklov} S.M.~Gusein-Zade, I.~Luengo, A.~Melle-Hern\'andez.
On the power structure over the Grothendieck ring of varieties and its applications.
Proc. Steklov Inst. Math. 258 (2007), no.1, 53--64.

\bibitem{GLMHigher} S.M.~Gusein-Zade, I.~Luengo, A.~Melle-Hern\'andez.
Higher order generalized Euler characteristics and generating series.
Preprint 2013,  ArXiv math.AG/1303.5574.

\bibitem{HH} F.~Hirzebruch, Th.~H\"ofer. 
On the Euler number of an orbifold.
Math. Ann. 286 (1990), no.1-3, 255--260. 


\bibitem{mac} I.G.~Macdonald. The Poincar\'e polynomial of a
symmetric product, Proc. Cambridge Philos. Soc. 58 (1962), 563--568.

\bibitem{T} H.~Tamanoi. Generalized orbifold Euler
characteristic of symmetric products and equivariant Morava $K$-theory.
Algebr. Geom. Topol. 1 (2001), 115--141.


\bibitem{Viro} O.Ya.~Viro. 
Some integral calculus based on Euler characteristic. 
Topology and geometry: Rohlin Seminar, 127--138, Lecture Notes in Math., 1346, Springer, Berlin, 1988.


\bibitem{WZ} W.~Wang, J.~Zhou. Orbifold Hodge numbers of wreath
product orbifolds. J. Geometry and Physics 38 (2001), 152--169.


\end{thebibliography}
\end{document}